\documentclass[12pt,reqno,a4paper]{article}
 \usepackage{amsmath,amssymb,amsthm,amsfonts,epsfig,enumerate}
 \usepackage{mathtools}
 \usepackage{multicol}

 \usepackage{color}
 \usepackage{enumerate}
 \usepackage{cite}
 \usepackage{hyperref}
 \hypersetup{linkcolor=blue, colorlinks=true ,citecolor = red}

\usepackage{amsmath,amssymb,amsthm,amsfonts,epsfig,enumerate}
\usepackage[utf8]{inputenc}
\usepackage{mathtools}
\usepackage{multicol}
\usepackage{pst-plot,pst-node}
\usepackage{mathrsfs,amssymb}
\usepackage{graphicx}
\usepackage{color}
\newtheorem{definition}{Definition}[section]
\newtheorem{theorem}[definition]{Theorem}
\newtheorem{example}[definition]{Example}

\newtheorem{proposition}[definition]{Proposition}
\newtheorem{lemma}[definition]{Lemma}

\numberwithin{equation}{section}
\DeclarePairedDelimiter\abs{\lvert}{\rvert}%
\DeclarePairedDelimiter\norm{\lVert}{\rVert}%

\makeatletter
\let\oldabs\abs
\def\abs{\@ifstar{\oldabs}{\oldabs*}}

\let\oldnorm\norm
\def\norm{\@ifstar{\oldnorm}{\oldnorm*}}
\newcommand{\al} {\alpha}

\newcommand{\De} {\Delta}

\newcommand{\om} {\omega}
\newcommand{\Om} {\Omega}
\newcommand{\la} {\lambda}

\newcommand{\noi} {\noindent}

\newcommand{\var} {\varepsilon}
\newcommand{\ra} {\rightarrow}

\newtheorem{rmk}[definition]{Remark}
\newcommand{\wra} {\rightharpoonup}

\DeclareMathAlphabet{\mathpzc}{T1}{pzc}{m}{it}

\def\w{{\widetilde w}}

\def\dt{{\rm d}t}

\def\sb2{{{\mathcal D}^{1,2}_0(B_1^c)}}

\def\w2r{{{ W}^{2,2}(\R^N)}}

\def\d2{{{\mathcal D}^{2,2}_0(\Om)}}
\def\cset{{\subset \subset }}

\def\C{{\mathcal C}}

\def\E{{\mathcal E}}

\def\R{{\mathbb R}}

\def\F{{\mathcal F}_{p,q}}
\def\({{\Big(}}
\def\){{\Big)}}

\def\ws2{{\F_{\frac{N}{2}}}}
\def\L2{{ L^{1,\;\infty}(\log L)^2}}

\def\l2{\mathcal M\log L}

\def\c1Loc{{\C_{loc}^1}}

\title{On weighted logarithmic-Sobolev \& logarithmic-Hardy inequalities}
\author{Ujjal Das\thanks{corresponding author} 
}

\date{}
\begin{document}
 \maketitle
\begin{abstract}
For $N \geq 3$ and $p \in (1,N)$, we look for $g \in L^1_{loc}(\mathbb{R}^N)$ that satisfies 
the following weighted logarithmic Sobolev inequality:
\begin{equation*}
 \int_{\mathbb{R}^N} g |u|^p \log |u|^p  \ dx \leq \gamma \log \left( C_{\gamma}  \int_{\mathbb{R}^N} |\nabla u|^p \ dx \right) \,,
\end{equation*} 
for all $u \in \mathcal{D}^{1,p}_0(\mathbb{R}^N)$ with $\int_{\mathbb{R}^N} g|u|^p=1$, for some $\gamma,C_{\gamma}>0$. 
For each $r \in(p,\frac{Np}{N-p}]$, we identify a Banach function space $\mathcal{H}_{p,r}(\mathbb{R}^N)$ such that the above inequality holds for $g \in \mathcal{H}_{p,r}(\mathbb{R}^N)$. For $\gamma > \frac{r}{r-p}$, we also find a class of $g$ for which the best constant $C_{\gamma}$ in the above inequality is attained in $\mathcal{D}^{1,p}_0(\mathbb{R}^N)$. Further, for a closed set $E$ with Assouad dimension $=d<N$ and $a \in (-\frac{(N-d)(p-1)}{p},\frac{(N-p)(N-d)}{Np}),$ we establish the following logarithmic Hardy inequality
\begin{equation*} 
\int_{\mathbb{R}^N} \frac{|u|^p}{|\delta_E|^{p(a+1)}} \log \left(\delta_E^{N-p-pa} |u|^p\right) \ dx \leq \frac{N}{p} \log \left(\text{C} \int_{\mathbb{R}^N} \frac{|\nabla u|^p}{|\delta_E^{pa}|} \ dx \right) \,,
\end{equation*}
for all $u \in C_c^{\infty}(\mathbb{R}^N)$ with $\displaystyle \int_{\mathbb{R}^N} \frac{|u|^p}{|\delta_E|^{p(a+1)}} =1,$ for some $\text{C}>0$, where $\delta_E(x)$ is the distance between $x$ and $E$. The second order extension of the logarithmic Hardy inequality is also obtained. 
 \end{abstract}

\maketitle
\section{Introduction} \label{intro}
For $N \geq 3$ and $p \in (1,N)$, the Sobolev inequalities states that
\begin{equation} \label{INTERHSINEQ}
 \left[ \int_{\R^N} |u|^{p^*} \ dx \right]^{\frac{1}{p^*}} \leq \text{C} \left[\int_{\R^N} |\nabla u|^p \ dx \right]^{\frac{1}{p}} \,,  \forall u \in \text{W}^{1,p}_0(\R^N) \,, 
\end{equation}
where $p^*=\frac{Np}{N-p}.$ Since $p^* \ra p$ as $N \ra \infty$, the gain in the integrability of $u$ disappears as $N \ra \infty.$ Though the Lorentz-Sobolev inequality \cite{Tartar}:
\begin{equation} \label{lorentzineq}
\displaystyle\int_0^{\infty} s^{\frac{p}{p^*}-1} |u^*(s)|^p \ ds  \leq \text{C} \displaystyle \int_{\R^N} |\nabla u|^p \ dx, \ \ \forall u \in \text{W}^{1,p}_0(\R^N)
\end{equation}
improves \eqref{INTERHSINEQ}, it has the same disadvantage.
In the study of quantum fields and hypercontractivity semi-groups, we often deal with an infinite dimensional space instead of $\R^N$. Thus, it is natural to look for an inequality that is dimension independent and  plays the role of Sobolev inequality. One such inequality is the Gross's {\it{logarithmic Sobolev inequality}} \cite{Gross}:
\begin{equation} \label{GROSSINEQ}
\int_{\R^N} |u|^2 \log |u|^2 \ d\mu \leq  \log \left(C \int_{\R^N} |\nabla u|^2 \ d\mu \right) \,,  \forall u \in C_c^{\infty}(\R^N) \,,
\end{equation}
where $\mu$ is a probability measure given by $d\mu(x)=(2 \pi)^{-\frac{N}{2}} e^{-\frac{|x|^2}{2}} \ dx$ and $\displaystyle\int_{\R^N} |u|^2 d\mu =1$.
This shows that $\text{H}^1_0(\R^N,d\mu)$ is embedded into the Orlicz space $L^2(Log L)(\R^N,d\mu)$ and the gain in the integrability of $u$ does not depend on $N$. An equivalent inequality involving Lebesgue measure is obtained in \cite{Weissler,Beckner}, namely
\begin{equation} \label{WiseIneq} 
\int_{\R^N} |u|^2 \log |u|^2 \ dx \leq \frac{N}{2} \log \left(C \int_{\R^N} |\nabla u|^2 \ dx \right) \,,  \forall u \in \text{H}^{1}_0(\R^N)
\end{equation}
with  $\displaystyle \int_{\R^N} |u|^2=1$. Unlike \eqref{GROSSINEQ}, the integrability of $u$ in \eqref{WiseIneq} (with respect to Lebesgue measure) is not dimension independent. This form of logarithmic Sobolev inequality arises in study of heat-diffusion semigroup, see \cite{Weissler}. In \cite{Del}, using the Gagliardo-Nirenberg inequalities, authors deduced the $L^p$ analogue of \eqref{WiseIneq}
\begin{equation} \label{LSINEQ}
\int_{\R^N} |u|^p \log |u|^p \ dx \leq \frac{N}{p} \log \left( \text{C} \int_{\R^N} |\nabla u|^p \ dx \right) \,,  \forall u \in \text{W}^{1,p}_0(\R^N)
\end{equation}
with  $\displaystyle \int_{\R^N} |u|^p=1$, for some $\text{C}>0$. The logarithmic Sobolev inequalities play an important role in the study of finite spin system \cite{Arthur}, non-linear McKean–Vlasov type PDE \cite{Malrieu}. For various other extensions and applications of these inequalities, we refer to \cite{Adams,Holley,Milman,Pick,Bobkov}. 

In this article, we are interested to find a {\it{weighted logarithmic Sobolev inequality}}. To be precise we look for a general class of weight functions $g \in L^1_{loc}(\R^N)$ so that
\begin{equation} \label{weightedLSINEQ}
 \int_{\R^N} g |u|^p \log |u|^p \ dx \leq \gamma \log \left(C_{\gamma}  \int_{\R^N} |\nabla u|^p \ dx \right) \,,  \forall u \in \mathcal{D}^{1,p}_0(\R^N)
 \end{equation}
with $\int_{\R^N} g|u|^p=1$ holds for some $\gamma,C_{\gamma}>0,$ where  $\mathcal{D}^{1,p}_0(\R^N)$ is the Beppo-Levi space defined as the completion of $C_c^{\infty}(\R^N)$ with respect to the norm $ \left[ \int_{\R^N} |\nabla u|^p\right]^{\frac{1}{p}}$. Indeed, the above inequality gives the logarithmic Sobolev type
inequalities involving the measure  gdx which is neither the Lebesgue measure nor a probability
measure. There are weighted logarithmic Sobolev inequalities where the weights are coupled with the gradient term in the right hand side of \eqref{GROSSINEQ}, see \cite{Cattiaux,Bin}. However, to the best of our knowledge the study of weighted logarithmic Sobolev inequality of the form \eqref{weightedLSINEQ} has not been done yet. In this article, we identify a Banach function space for $g$ so that the weighted logarithmic Sobolev inequality \eqref{weightedLSINEQ} holds. 
In order to do so, we first look for a class of weights $g \in L^1_{loc}(\R^N)$ for which the following weighted Hardy-Sobolev inequality holds:
\begin{equation} \label{weightedHardy}
\left[\int_{\R^N} |g| |u|^r \ dx \right]^{\frac{1}{r}} \leq \text{C}
\left[\int_{\R^N} |\nabla u|^p \ dx \right]^{\frac{1}{p}} \,,  \forall u \in \mathcal{D}^{1,p}_0(\R^N) \,.
\end{equation}
Such weighted Hardy-Sobolev inequality has been discussed in \cite{Visciglia,New,Allegretto,anoop}. A similar study has been done for the Hardy-Rellich inequalities too, see \cite{biharmonic,Zographopoulos}. In this context, using the notion of p-capacity \cite{Mazya2}, Maz$'$ya has provided a necessary and sufficient condition on $g$ so that \eqref{weightedHardy} holds. Let us recall that, for $F \cset \R^N$, the p-capacity of $F$ with respect to $\R^N$ is defined as
$$\text{Cap}_p(F)= \inf \left\{\int_{\R^N} |\nabla u|^p : u \in \mathcal{N}(F) \right\} \,,$$
where $\mathcal{N}(F)=\{u \in \mathcal{D}^{1,p}_0(\R^N): u \geq 1 \ \mbox{on} \ F\}$. Now, for $1<p \leq r \leq p^*$ we define
\begin{eqnarray*}
\|g\|_{p,r} &=& \displaystyle\sup_{F \cset \R^N} \left\{\frac{\displaystyle \int_{F}|g| }{[\text{Cap}_p(F)]^{\frac{r}{p}}} \right\} \,, \\
\mathcal{H}_{p,r}(\R^N) &=& \left\{g \in L^1_{loc}(\R^N): \|g\|_{p,r}<\infty \right\} \,.
\end{eqnarray*}
In \cite{Mazya2}, Maz$'$ya has shown that \eqref{weightedHardy} holds for $r \in [p,p^*]$ if and only if $g \in \mathcal{H}_{p,r}(\R^N)$. In fact, $\mathcal{H}_{p,r}(\R^N)$ is a Banach function space equipped with the norm $\|.\|_{p,r}$. We also provide some examples of classical functions spaces in $\mathcal{H}_{p,r}(\R^N)$ (Remark \ref{521}). Now one may anticipate \eqref{weightedLSINEQ} for $g \in \mathcal{H}_{p,r}(\R^N)$. Indeed, we have the following result.
\begin{theorem} \label{weightthm}
Let $N \geq 3$, $p \in (1,N)$ and $r \in (p,p^*]$. If $g \in \mathcal{H}_{p,r}(\R^N)$, then
\begin{equation} \label{WLS}
\int_{\R^N} g |u|^p \log |u|^p \ dx \leq  \frac{r}{r-p} \log  \left( \text{C}_H \|g\|_{p,r}^{\frac{p}{r}} \ \int_{\R^N}  |\nabla u|^p \ dx \right) \,,
\end{equation}
for all $u \in \mathcal{D}^{1,p}_0(\R^N)$ with $\int_{\R^N} g |u|^p=1$, where $\text{C}_H= p^p (p-1)^{(1-p)}.$
\end{theorem}
\noi Therefore, \eqref{weightedLSINEQ} holds for $\gamma \geq  \frac{r}{r-p}$. Let $\text{C}_B(g,\gamma)$ be the best constant in \eqref{weightedLSINEQ}. Then,
\begin{equation*}
\frac{1}{\text{C}_B(g,\gamma)} = \displaystyle \inf \left \{\frac{\int_{\R^N} |\nabla u|^p}{e^{\frac{1}{\gamma} \left(\int_{\R^N} |g||u|^p \log |u|^p \right)}} : u \in \mathcal{D}^{1,p}_0(\R^N), \int_{\R^N} |g||u|^p =1 \right \} \,.    
\end{equation*} 
It is clear that $\text{C}_B(g,\gamma) \leq \text{C}_H \|g\|_{p,r}$ for $g \in \mathcal{H}_{p,r},$ $\gamma \geq \frac{r}{r-p}$.
It is natural to look for a class of weights $g$ and values of $\gamma$ for which $\text{C}_B(g,\gamma)$ is attained in $\mathcal{D}^{1,p}_0(\R^N).$ In this context, we define the following closed sub-space
$$\mathcal{F}_{p,r}(\R^N)=\overline{C_c^{\infty}(\R^N)} \ \text{in} \ \mathcal{H}_{p,r}(\R^N ) \,.$$
Now we state our result.
\begin{theorem} \label{Bestthm}
Let $N \geq 3$, $p \in (1,N)$ and $r \in (p,p^*]$. If $g \in \mathcal{H}_{p,r}(\R^N) \cap \mathcal{F}_{p,p}(\R^N)$ and $\gamma > \frac{r}{r-p}$, then
$\text{C}_B(g,\gamma)$ is attained in $\mathcal{D}^{1,p}_0(\R^N).$
\end{theorem}

Next, we recall the classical Hardy inequality \cite{Opic,Jesper}
\begin{equation} \label{CHINEQ}
\int_{\R^N} \frac{|u|^p}{|x|^p} \ dx \leq \left(\frac{p}{N-p}\right)^p \int_{\R^N} |\nabla u|^p \ dx \,, \forall u \in C_c^{\infty}(\R^N)\,.
\end{equation}
Many improvements and generalizations of Hardy inequality are available in literature, for instance see \cite{Visciglia,Filippas,Vazquez,Adimurthi}. We refer to \cite{Kufner} for a comprehensive review on this topic. A variant of \eqref{CHINEQ} involving the distance from the boundary of the domain instead of a distance to a point singularity has been studied extensively, see \cite{Alvino,Dupaigne,Tintarev,Junfang,Gkikas}. If $\Om$ is open, convex with a smooth boundary $\partial \Om$, then the following variant of Hardy inequality \cite{Sobolevskii}
\begin{equation} \label{distthm} 
\int_{\Om} \frac{|u|^p}{|\delta_{\partial \Om}(x)|^p} \ dx \leq \left(\frac{p}{p-1}\right)^p \int_{\Om} |\nabla u|^p \ dx \,,  \forall u \in C_c^{\infty}(\Om) \,,
\end{equation}
holds for $p \in (1,\infty)$, where $\delta_A(x)$ denotes the distance of $x$ from a closed set $A$. Later, the convexity assumptions are relaxed in \cite{Barbatis,Lewis}. Further, recently in \cite{Juha,Dyda} authors have considered the distance from a general closed set $E$ in $\R^N$ instead of $\partial \Om$.
Under certain conditions on the `Assouad dimension' of $E$ (see Section \ref{label} for the precise definition), they obtained the following global Hardy inequality:
\begin{equation} \label{DISTHINEQ}
\int_{\R^N} \frac{|u|^p}{|\delta_E(x)|^p} \ dx \leq C \int_{\R^N} |\nabla u|^p \ dx, \ \ \forall u \in C_c^{\infty}(\R^N) \,,
\end{equation}
for $1<p<\infty$. Notice that, for $E=\{0\}$ and  $E=\partial \Om  $, \eqref{DISTHINEQ} corresponds to \eqref{CHINEQ} and \eqref{distthm} respectively. In \cite{Pino} authors obtained the following {\it{logarithmic Hardy inequality}}:
\begin{equation} \label{INTERLSHINEQ}
  \int_{\R^N} \frac{|u|^2}{|x|^2} \log \left(|x|^{N-2} |u|^2 \right) \ dx \leq \frac{N}{2} \log \left(C \int_{\R^N} |\nabla u|^2 \ dx \right) \,, 
\end{equation}
for all $u \in C_c^{\infty}(\R^N)$ with $\displaystyle \int_{\R^N} \frac{|u|^2}{|x|^2}=1$. Notice that, the integrals in \eqref{INTERLSHINEQ} are scale invariant which distinguishes them from logarithmic Sobolev inequalities \eqref{LSINEQ}. 
The Caffarelli–Kohn–Nirenberg type inequality \cite{CKN}
played a key role in \cite{Pino} to achieve \eqref{INTERLSHINEQ}. In a recent work \cite{Dyda} authors proved a variant of Caffarelli–Kohn–Nirenberg type inequalities involving $\delta_E$ under certain restriction on Assouad dimension of $E$. This facilitates us to achieve the logarithmic Hardy inequality involving $\delta_E$ as stated in the following theorem. 
\begin{theorem} \label{genthm}
Let $N \geq 3$, $p \in (1,N)$ and $E$ be a closed set in $\R^N$ with ${\text{Assouad dimension}}$ $=d<N$. Then, for $a \in (-\frac{(N-d)(p-1)}{p},\frac{(N-p)(N-d)}{Np}),$ there exists $\text{C}>0$ such that 
\begin{equation} \label{GINTERDLSHINEQ}
\int_{\R^N} \frac{|u|^p}{|\delta_E|^{p(a+1)}} \log \left(\delta_E^{N-p-pa} |u|^p\right) \ dx \leq \frac{N}{p} \log \left(\text{C} \int_{\R^N} \frac{|\nabla u|^p}{|\delta_E^{pa}|} \ dx \right) \,,
\end{equation}
for all $u \in C_c^{\infty}(\R^N)$ with $\displaystyle \int_{\R^N} \frac{|u|^p}{|\delta_E|^{p(a+1)}} =1.$
\end{theorem}
\noi  Notice that, \eqref{INTERLSHINEQ} follows from \eqref{GINTERDLSHINEQ} by taking $E=\{0\}$, $p=2$ and $a=0$.
We also have the following extension of \eqref{GINTERDLSHINEQ}.
\begin{theorem} \label{higherthm}
Let $N \geq 3$, $p \in (1,\frac{N}{2})$ and $E$ be a closed set in $\R^N$ with ${\text{Assouad dimension}}$ $=d<\frac{N(N-2p)}{(N-p)}$. Then, for each $a \in (1-\frac{(N-d)(p-1)}{p},\frac{(N-p)(N-d)}{Np})$ there exists $\text{C}>0$ such that 
\begin{equation} \label{HIGHERINTERDLSHINEQ}
\int_{\R^N} \frac{|u|^p}{|\delta_E|^{p(a+1)}} \log \left(\delta_E^{N-p-pa} |u|^p\right) \ dx \leq \frac{N}{p} \log \left(\text{C} \int_{\R^N} \frac{|\nabla^2 u|^p}{|\delta_E^{(a-1)p}|} \ dx \right) \,,
\end{equation}
for all $u \in C_c^{\infty}(\R^N)$ with $\displaystyle \int_{\R^N} \frac{|u|^p}{|\delta_E|^{p(a+1)}} =1,$ where $|\nabla^2 u|^2=\sum_{i,j=1}^N (\frac{\partial^2 u}{\partial x_i \partial x_j})^2$.
\end{theorem}

So far we have discussed about logarithmic version of Sobolev inequality and Hardy inequalities. In the same spirit, we establish a logarithmic version of Lorentz-Sobolev inequality as stated below. 
\begin{theorem} \label{LoSOthm}
Let $N \geq 3$ and $p \in (1,N)$. Then, for all $u \in \mathcal{D}^{1,p}_0(\R^N)$ with $\|u\|_{L^{p^*,p}}=1$, there exists $\text{C}>0$ such that $$\int_0^{\infty} s^{\frac{p}{p^*}-1} |u|^p \log (s^{1-\frac{p}{N}} |u|^p ) \ ds \leq \frac{N}{p} \log \left(\text{C} \int_{\R^N} |\nabla u|^p \ dx \right),.$$
\end{theorem}

The rest of this article is organised as follows. In Section \ref{prelim}, we recall the notion of Assouad dimension of a set in $\R^N$ and we briefly discuss the Banach function spaces $\mathcal{H}_{p,r}(\R^N)$, $\mathcal{F}_{p,r}(\R^N)$ and some known results which are essential for this article. We prove Theorem \ref{weightthm}, Theorem \ref{Bestthm} in Section \ref{3}. Section \ref{proofs} is devoted to prove Theorem \ref{genthm}, Theorem \ref{higherthm} and some important remarks on these theorems. Further, we prove the Theorem \ref{LoSOthm} in Section \ref{4}.

\section{Preliminaries} \label{prelim}
In this section, we recall the notion of Assouad dimension of a set in $\R^N$ and briefly discuss some of its properties. Further, we state some of the known results which are essential for this manuscript.
\subsection{Assouad dimension} \label{label}
For a general subset $E$ of $\R^N$, $\eta (E,r)$ denotes the minimal number
of open balls of radius $r$ with centers in $E$ that are needed to cover the set $E$. The Assouad dimension is denoted by $\text{dim}_A(E)$ and is defined as follows.
\begin{eqnarray*}
 \text{dim}_A(E)= \inf \Bigg \{\la \geq 0: \mbox{there exists} \ C_{\la} > 0 \ \text{so that} 
 \ \eta(E \cap B_R(x),r) \leq  C_{\la} \left(\frac{r}{R} \right)^{-\la}, \\
 \forall x \in E, 0<r<R< \text{diam}(E)  \displaystyle \Bigg \} \,.
\end{eqnarray*}
In the case when $\text{diam}(E)=0$, we remove the restriction $R<\text{diam}(E)$ from above definition. By this convention one can see that, if $E=\{x_0\}$ for some $x_0 \in \R^N$ then $\text{dim}_A(E)=0.$
We refer to \cite{Assouad} for a historical background of Assouad dimension and its basic properties. More recent results on this can be found in \cite{Fraser}. Here, we enlist some of its basic properties in the following proposition, for a proof see \cite{Assouad}.

\begin{proposition} \rm Let $(X,d)$ be a metric space. Then the following statements are true:
\begin{enumerate}
 \item[(i)] $X$ has finite Assouad dimension if and only if it is a doubling space, i.e. there exists a finite constant $C>0$ such that every ball of radius $r$ can be covered by no more than $C$ balls of radius $\frac{r}{2}$.
 \item[(ii)] If $Y \subset X$, then $\text{dim}_A(Y) \leq \text{dim}_A(X)$. Equality holds if $\overline{Y}=X.$
 \item[(iii)] $\text{dim}_H(X) \leq \text{dim}_A(X),$ where  $\text{dim}_H$ denotes the Hausdroff dimension of $X.$
 \item[(iv)] Let $E \subset \R^N$. Then $\text{dim}_A(E) <N$ if and only if
$E$ is porous in $\R^N$ i.e. there
is a constant $\al (0,1)$ such that for every $x \in E$  and all $0<r<\text{diam}(E)$ there exists a
point $y \in \R^N$ such that $B_{\al r}(y) \subset B_r(x) \setminus E.$
 \end{enumerate}
 \end{proposition}
\begin{rmk} \label{boundaryporous} \rm
Notice that, for $x \in \partial B_1$ and $r \in (0,2)$ we can find $y \in B_r(x)$ such that $B_{\frac{r}{4}}(y) \subseteq B_r(x) \setminus  \partial B_1.$ Hence,
by the definition of porousity, it follows that the boundary of unit ball is porous in $\R^N$. Hence, $\text{dim}_A(\partial B_1) <N.$ Similarly, it can be seen that $\R^{N-1} \times \{0\}$ is porous in $\R^N$. Hence,
$\text{dim}_A(\R^{N-1} \times \{0\}) <N.$
\end{rmk}

\subsection{Lorentz spaces}
The Lorentz spaces are refinements of the usual Lebesgue spaces and introduced by Lorentz  in\cite{Lorentz}. For more details on Lorentz spaces and related results, we refer to the book \cite{EdEv}. 

Given a measurable function $f:\R^N \mapsto \R$ and $s>0$, we define
$E_f(s)=\{x: |f(x)|>s \}$ and the distribution function $\alpha_f$ of $f$ is defined as 
\begin{eqnarray*}
\alpha_f(s) &: =&
 \big\vert E_f(s) 
 \big\vert, \, \mbox{ for } s>0,
\end{eqnarray*}
where $|A|$ denotes the Lebesgue measure of a set $A\subseteq \R^N.$ We define the {\it one dimensional decreasing rearrangement} $f^*$ of $f$ as below: 
\begin{align*}
f^*(t):= \begin{cases*} \operatorname{ess}\ \sup f, \ \ t =0\\ \inf \{s>0 \, : \, \alpha_f(s) < t \}, \; t>0.   \end{cases*}
\end{align*}
For $(p,q) \in [1,\infty)\times[1,\infty]$ we consider the following quantity:
\begin{align*} 
 \|f\|_{L^{p,q}} := \norm{t^{\frac{1}{p}-\frac{1}{q}} f^{*} (t)}_{{L^q((0,\infty))}}
=\left\{\begin{array}{ll}
         \left(\displaystyle\int_0^\infty \left[t^{\frac{1}{p}-\frac{1}{q}} {f^{*}(t)}\right]^q \dt \right)^{\frac{1}{q}};\; 1\leq q < \infty, \vspace{4mm}\\ 
         \displaystyle\sup_{t>0}t^{\frac{1}{p}}f^{*}(t);\; q=\infty.
        \end{array} 
\right.
\end{align*}
The Lorentz space $L^{p,q}(\R^N)$ is defined as
\[ L^{p,q}(\R^N) := \left \{ f:\R^N \mapsto \R \ \text{measurable},  \  \|f\|_{L^{p,q}}<\infty \right \} \,.\]
$ \|f \|_{L^{p,q}}$ is  a complete quasi-norm on $L^{p,q}(\Om).$ 
 We recall some properties of Lorentz spaces.
 \begin{proposition} \label{LorentzProp} 
 
 \begin{enumerate} [(i)]
 \item {\bf{Inclusion in primary index:}} Let $p,q,r,s \in [1,\infty]$ and $r < p$. Then
 $L^{p,q}(\R^N) \hookrightarrow L_{loc}^{r,s}(\R^N).$
 \item {\bf{Inclusion in secondary index:}} Let $p,q,r \in [1,\infty]$ and $q \leq r.$ Then $L^{p,q} \hookrightarrow L^{p,r}(\R^N).$
 \item {\bf{Holder's Inequality:}} For $(f, g) \in L(p_1 , q_1 ) \times
L(p_2 , q_2 )$ and $(p, q) \in (1, \infty) \times [1, \infty]$ such that $\frac{1}{p} = \frac{1}{{p_1}} + \frac{1}{{p_2}} ,
\frac{1}{q} \leq  \frac{1}{{q_1}} + \frac{1}{{q_2}} $, then
$$\|f g\|_{L^{p,q}} \leq  C \|f \|_{L^{p_1 ,q_1}} \|g\|_{L^{p_2 ,q_2 }} \,,$$
where $C$ depends only on $p$.
 \end{enumerate}
  \end{proposition}
\begin{proof}
See \cite{EdEv} for the proofs.
\end{proof}  
  
Next we state the Lorentz-Sobolev inequality obtained in \cite{Tartar}.  
\begin{theorem} \label{LoSo}
 Let $N \geq 3$ and $p \in (1,N)$. Then $\mathcal{D}^{1,p}_0(\R^N) \hookrightarrow L^{p^*,p}(\R^N)$ i.e. there exists
 $C>0$ such that
 $$\displaystyle\int_0^\infty \left[s^{\frac{1}{p^*}-\frac{1}{p}} {u^{*}(s)}\right]^p \ ds 
 \leq C \int_{\R^N} |\nabla u|^p \ dx \,,  \forall u \in \mathcal{D}^{1,p}_0(\R^N) \,.$$
\end{theorem}

\subsection{Brezis-Lieb lemma}
Let $J:\R \mapsto \R$ be a continuous function with $J(0)=0$ such that,
for every $\epsilon >0$ there exist
two continuous, non-negative functions $\phi_{\var},\psi_{\var}$ satisfying
\begin{equation} \label{BLcond}
|J(a+b)-J(a)| \leq \epsilon \phi_{\var}(a) +  \psi_{\var}(b) \,, \forall a,b \in \R \,.   
\end{equation}
Now we state a lemma proved by Brezis and Lieb in \cite{BrezisLeib}.
\begin{lemma} \label{BLlemma}
 Let $J:\R \mapsto \R$ satisfies \eqref{BLcond} and $f_n=f+g_n$ be a sequence of measurable functions on $\Om$ to $\R$ such that
\begin{enumerate}[(i)]
    \item $g_n \ra 0$ a.e.,
    \item $J(f) \in L^1(\Om)$,
    \item $\displaystyle \int_{\Om} \phi_{\var}(g_n(x)) \ d\mu(x) \leq C <\infty$, for some $C>0$ independent of $n,\epsilon$,
    \item $\displaystyle \int_{\Om} \psi_{\var}(f(x)) \ d\mu(x) <\infty$, for all $\epsilon >0.$
\end{enumerate}
Then
$$\lim_{n \ra \infty} \int_{\Om} |J(f+g_n) -J(g_n)-J(f)| \ d\mu=0 \,.$$
\end{lemma}
\begin{example} \label{Functional} \rm
\noi (A). If $J$ is convex on $\R$, then $J$ satisfies \eqref{BLcond}. In particular, if $J(t)=|t|^p$;
 $p \in (1,\infty)$, then \eqref{BLcond} is valid with $\phi_{\var}(t)=|t|^p$ and $\psi_{\var}(t)=C_{\var} |t|^p$ for sufficiently large $C_{\var}$, see \cite{BrezisLeib}.

\noi (B). Let $J(t)=t^p\log t$, for $t \geq 0$. Then $J$ is continuous and $J(0)=0$. Further, for $a,b \geq 0$, using mean value theorem we obtain
 \begin{eqnarray*}
 |J(a+b)-J(a)| &\leq& (a+b)\left[(a+b)^{p-1} + p(a+b)^{p-1} \log (a+b)\right] \\
 & \leq & \begin{cases*} (a+b)^p \,,  \ \ \ \ \ \ \ \ \  \ \ \ \quad \text{if} \ a+b \leq 1 \,, \\
 (p+1)(a+b)^{p^*} \,, \quad \text{if} \ a+b \geq 1 \,.
 \end{cases*}
 \end{eqnarray*}
Thus, it follows from previous example that $J$ satisfies \eqref{BLcond}. 
\end{example}

\subsection{Some inequalities}
We recall some inequalities here that are essential for the development of this article.
\subsubsection{A logarithmic Hardy inequality}
The following inequality is obtained in \cite{Pino} which is one of the main motivations for this article.
\begin{theorem} \label{Tertikasthm}
 Let $N \geq 3$. Then, for $a < (\frac{N-2}{2})$ there exists $\text{C}>0$ such that 
 \begin{equation} \label{2.7} 
 \int_{\R^N} \frac{|u|^2}{|x|^{2(a+1)}} \log [|x|^{N-2-2a}|u|^2]  \leq \frac{N}{2} \log \left( \text{C} \int_{\R^N}   \frac{|\nabla u|^2}{|x|^{2a}} \right) \,, 
\end{equation}
for all $u \in C_c^{\infty}(\R^N)$ with $\displaystyle \int_{\R^N} \frac{|u|^2}{|x|^{2(a+1)}}=1$.
\end{theorem}
\subsubsection{Caffarelli-Kohn-Nirenberg type inequalities involving distance function}
A variant of C-K-N inequality involving distance function is obtained in \cite{Dyda} that is stated in the following lemma.
\begin{lemma} \label{Dydalemma}
Let $N \geq 3$ and $1<p\leq q \leq p^*$. Further, for a closed set $E$ in $\R^N$ and $\beta \in \R$ assume 
$$\text{dim}_A(E)<\frac{q}{p}\left(N-p+\beta \right) \ \quad \mbox{and} \quad \text{dim}_A(E)< N-\frac{\beta}{p-1} \,.$$
Then, there exists $\text{C}>0$ such that 
\begin{equation*}\label{DYDAINEQ} 
\left[\int_{\R^N} \frac{|u|^q}{\delta_{E}^{N-\frac{q}{p}(N-p+\beta)}}  \right]^{\frac{1}{q}}  \leq C \left[\int_{\R^N} |\nabla u|^p \delta_{E}^{\beta} \right]^{\frac{1}{p}} \,, u \in C_c^{\infty}(\R^N)\,.
\end{equation*}
\end{lemma}
\begin{proof} 
For a proof we refer to Remark 6.2 of \cite{Dyda}.
\end{proof}
\subsubsection{Maz$'$ya's Condition}
Maz$'$ya's provide a necessary and sufficient condition (Theorem 8.5 of \cite{Mazya2}) on $g$ so that \eqref{weightedHardy} holds. This condition helps us to identify a class of weights for which weighted Hardy-Sobolev type inequalities holds.
\begin{lemma} \label{Mazyalemma}
 Let $1<p\leq r \leq p^*$. Then $g \in \mathcal{H}_{p,r}(\R^N)$ if and only if 
 \begin{equation} \label{mazyaineq}
 \left[\int_{\R^N} |g| |u|^r \right]^{\frac{p}{r}}  \leq \text{C}_H \|g\|_{p,r}^{\frac{p}{r}} \left[ \int_{\R^N} |\nabla u|^p \right], \ \ \forall u \in C_c^{\infty}(\R^N) \,,
 \end{equation}
where $\text{C}_H= p^p (p-1)^{(1-p)}$. 
\end{lemma}

\subsection{The function spaces $\mathcal{H}_{p,r}(\R^N)$ and $\mathcal{F}_{p,r}(\R^N)$} \label{funspace}
\begin{definition} \label{BFC} \rm
A normed linear space $(X,\norm{.}_X)$ of measurable functions is called a Banach function space if the following conditions are satisfied:
\begin{enumerate}[(i)]
\item $\norm{f}_X = \norm{|f|}_X$, for all $f \in X$,
\item if $ (f_n)$ is a non-negative sequence of function in $X$, increases to $f$, then $\norm{f_n}_X$ increases to $ \norm{f}_X$,
\item if $ E \subseteq \Om$ with $|E| < \infty$, then $\chi_E \in X$,
\item for all $E \subseteq \Om$ with $|E| < \infty $, there exist $C_E >0$ such that $ \int_E |f| \leq C_E \norm{f}_X .$
\end{enumerate}
  \noi The norm $\norm{.}_X$ is called a Banach function norm on $X.$
\end{definition}
For $r \in [p,p^*]$, it can be easily verified that
$\mathcal{H}_{p,r}(\R^N) = \left\{g \in L^1_{loc}(\R^N): \|g\|_{p,r}<\infty \right\}$ is a Banach function space with respect to the norm
\begin{eqnarray*}
\|g\|_{p,r} &=& \displaystyle\sup_{F \cset \R^N} \left\{\frac{\displaystyle \int_{F}|g| }{[\text{Cap}_p(F)]^{\frac{r}{p}}} \right\} \,.
\end{eqnarray*}
We also define $\mathcal{F}_{p,r}(\R^N)=\overline{C_c^{\infty}(\R^N)}$ in $\mathcal{H}_{p,r}(\R^N).$

The forthcoming proposition ensures that the spaces considered in Theorem \ref{weightthm} and Theorem \ref{Bestthm} contain certain Lebesgue spaces.
\begin{proposition} \label{functionspace}
Let $p \in (1,N)$ and $r \in [p,p^*]$. Then
\begin{enumerate}[(i)]
\item $L^{\frac{p^*}{p^*-r}}(\R^N) \subseteq \mathcal{H}_{p,r}(\R^N)$ ,
\item $L^1 \cap L^{\frac{p^*}{p^*-r}}(\R^N) \subseteq \mathcal{H}_{p,r}(\R^N) \cap \mathcal{F}_{p,p}(\R^N)$.
\end{enumerate}
\end{proposition}
\begin{proof} 
(i). Let $g \in L^{\frac{p^*}{p^*-r}}(\R^N)$ for some $r \in (p,p^*].$ Notice that, for $F \cset \R^N,$
$$\displaystyle \frac{\int_{F} |g|}{[\text{Cap}_p(F)]^{\frac{r}{p}}}  \leq \|g\|_{L^\frac{p^*}{p^*-r}} \left[\frac{|F|^{\frac{r}{p^*}}}{[\text{Cap}_p(F)]^{\frac{r}{p}}}\right] \,. $$
Since $|F| \leq C [\text{Cap}_p(F)]^{\frac{N}{N-p}}$ ( Theorem 4.15, \cite{EvansGa}), it follows that 
$$ \|g\|_{p,r} = \sup_{F \cset \R^N} \frac{\int_{F} |g|}{[\text{Cap}_p(F)]^{\frac{r}{p}}}   \leq C \|g\|_{L^\frac{p^*}{p^*-r}} < \infty  \,,$$
i.e. $g \in \mathcal{H}_{p,r}(\R^N).$ Hence, $L^{\frac{p^*}{p^*-r}}(\R^N) \subseteq \mathcal{H}_{p,r}(\R^N).$

\noi (ii). 
Let $g \in L^1 \cap L^{\frac{p^*}{p^*-r}}(\R^N)$. Then part (i) implies $g \in \mathcal{H}_{p,r}(\R^N)$. For $r=p$, it follows from Proposition 17 of \cite{New} that $g \in \mathcal{F}_{p,p}(\R^N)$. For $r>p$ we will use Theorem 31 of \cite{New} to show that $g \in \mathcal{F}_{p,p}(\R^N).$ Since $1<\frac{N}{p}<\frac{p^*}{p^*-r}$, we have $g \in L^{\frac{N}{p}}(\R^N)$ and hence, $g \in \mathcal{H}_{p,p}(\R^N)$. Now, for any $x \in \R^N$ and $t>0$,
\begin{equation*}
 \displaystyle \frac{\int_{B_{t}(x)} |g|}{\text{Cap}_p(B_{t}(x))} \leq \frac{\left(\int_{\R^N} |g|^{\frac{p^*}{p^*-r}}\right)^{\frac{p^*-r}{p^*}} \left(\om_Nt^N\right)^{\frac{r}{p^*}}}{N\om_N \left(\frac{N-1}{p-1}\right)^{p-1}t^{N-p}} \leq \text{C} t^{\left(\frac{r}{p} -1\right) (N-p)} \,.   
\end{equation*}
Thus, 
\begin{equation} \label{60}
\lim_{t \ra 0} \left[\sup_{F \cset \R^N} \frac{\int_{F \cap B_t(x)} |g|}{\text{Cap}_p(B_{t}(x))} \right] =0 \,.    
\end{equation}
Similarly, one can obtain
\begin{equation*}
 \displaystyle \frac{\int_{F \cap B_{t}(0)^c} |g|}{\text{Cap}_p(F)} \leq 
 \begin{cases*}
 \int_{B_{t}(0)^c} |g| \,, \quad \ \ \ \ \ \ \ \ \ \ \ \ \ \ \ \text{if} \ \ \text{Cap}_p(F) \geq 1 \\
 \left[\int_{B_{t}(0)^c} |g|^{\frac{p^*}{p^*-r}}\right]^{\frac{p^*-r}{p^*}} \,, \quad \ \text{if} \  \ \text{Cap}_p(F) < 1 \,.
 \end{cases*}
\end{equation*}
This infers
\begin{equation} \label{61}
\lim_{t \ra \infty} \left[\sup_{F \cset \R^N} \frac{\int_{F \cap B_t(x)} |g|}{\text{Cap}_p(B_{t}(x))} \right] =0 \,.     
 \end{equation}
Since \eqref{60}, \eqref{61} hold and $g \in \mathcal{H}_{p,p}(\R^N)$, it follows from Theorem 31 of \cite{New} that $g \in \mathcal{F}_{p,p}(\R^N).$ Therefore, $L^1 \cap L^{\frac{p^*}{p^*-r}}(\R^N) \subseteq \mathcal{H}_{p,r}(\R^N) \cap \mathcal{F}_{p,p}(\R^N).$ 
\end{proof}
\begin{rmk} \label{521}
 \rm Indeed, $L^{\frac{p^*}{p^*-r},\infty}(\R^N) \subseteq \mathcal{H}_{p,r}(\R^N);$ $r \in [p,p^*].$
For instance, let $g \in L^{\frac{p^*}{p^*-r},\infty}(\R^N).$ Then, using part (iii) of Proposition \ref{LorentzProp} we obtain
\begin{eqnarray*}
 \int_{\R^N} |g||u|^r \leq \|g\|_{L^{\frac{p^*}{p^*-r},\infty}} \||u|^r\|_{L^{\frac{p^*}{r},1}} &=& \|g\|_{L^{\frac{p^*}{p^*-r},\infty}} \|u\|_{L^{p^*,r}}^r , \ \ \forall u \in \mathcal{D}^{1,p}_0(\R^N) \,.
\end{eqnarray*}
Theorem \ref{LoSo} and part (ii) of Proposition \ref{LorentzProp} ensure that $\mathcal{D}^{1,p}_0(\R^N) \hookrightarrow L^{p^*,p}(\R^N) \hookrightarrow L^{p^*,r}(\R^N)$). Thus, the above inequality gives
\begin{eqnarray*}
 \int_{\R^N} |g||u|^r \leq  \text{C} \|g\|_{\frac{p^*}{p^*-r},\infty} \left[ \int_{\R^N} |\nabla u|^p\right]^{\frac{r}{p}}, \ \ \forall u \in\mathcal{D}^{1,p}_0(\R^N) \,.
\end{eqnarray*}
Hence,  $g \in \mathcal{H}_{p,r}(\R^N)$ (by Lemma \ref{Mazyalemma}). 
\end{rmk}
Now we define an associate space of $\mathcal{H}_{p,p}(\R^N)$ as follows. Let
\begin{eqnarray*}
\|u\|_{p,p}' &=& \sup \left\{\int_{\R^N} |g||u|^p: g \in \mathcal{H}_{p,p}(\R^N) \ \text{and} \ \|g\|_{p,p} \leq 1 \right\} \,. \end{eqnarray*}
\noi Then we define
\begin{eqnarray*}
\E_{p,p}(\R^N) &=& \{u:\R^N \mapsto \R: u \ \text{is measurable} , \ \|u\|_{p,p}' < \infty \} \,.
\end{eqnarray*}
Indeed, $\E_{p,p}(\R^N)$ is a Banach function space. Next we state the following embedding proved in Theorem 20 of \cite{New}.
\begin{proposition} \label{associate}
Let $N \geq 3$ and $p \in (1,N)$. Then $\mathcal{D}^{1,p}_0(\R^N) \hookrightarrow \E_{p,p}(\R^N)$ i.e. there exists $C>0$ such that
$$\|u\|'_{p,p} \leq C \left[\int_{\R^N} |\nabla u|^p \right]^{\frac{1}{p}} \,, \forall u \in \mathcal{D}^{1,p}_0(\R^N) \,.$$
\end{proposition}

\section{Weighted logarithmic Sobolev inequality} \label{3}
In this section, we prove Theorem \ref{weightthm} and Theorem \ref{Bestthm}.
We provide some example of functions and make some remarks related to these theorems.

\noi{\bf{Proof of Theorem \ref{weightthm}:}}
Let $g \in \mathcal{H}_{p,r}(\R^N)$ for some $r\in (p,p^*]$. For $q \in [p,r)$, take $k = p\frac{r-q}{r-p}$. Now, using Holder's inequality we estimate the following integral:
\begin{eqnarray*}
 \int_{\R^N} |g||u|^q \ dx &=& \int_{\R^N} |g|^{\frac{k}{p}}|u|^k  |g|^{\frac{p-k}{p}}|u|^{q-k} \ dx \\
 &\leq& \left[ \int_{\R^N} |g||u|^p \ dx \right]^{\frac{k}{p}} \left[ \int_{\R^N} |g||u|^{\frac{p(q-k)}{p-k}} \ dx \right]^{\frac{p-k}{p}} \\
& = &
  \left[ \int_{\R^N} |g||u|^p \ dx \right]^{\frac{r-q}{r-p}} \left[ \int_{\R^N} |g||u|^{r} \ dx \right]^{\frac{q-p}{r-p}}, \ \ \forall u \in \mathcal{D}^{1,p}_0(\R^N) \,.
\end{eqnarray*}
For small $t>0$, take $q=p+t$ in the above inequality to obtain
\begin{eqnarray*}
 \int_{\R^N} |g||u|^{p+t} \ dx
 \leq \left[ \int_{\R^N} |g|u|^p \ dx \right]^{\frac{r-(p+t)}{r-p}} \left[ \int_{\R^N} |g||u|^{r} \ dx \right]^{\frac{(p+t)-p}{r-p}}.
\end{eqnarray*}
Notice that, for $t=0$ equality holds. Thus
\begin{equation} \label{mainineq2}
 \displaystyle \int_{\R^N} \frac{1}{t} \big[|g||u|^{p+t}- |g||u|^p \big] \ dx
\leq \frac{1}{t} \left[\text{A}_1^{\frac{r-(p+t)}{r-p}} \text{B}_1^{\frac{(p+t)-p}{r-p}} - \text{A}_1^{\frac{r-p}{r-p}} \text{B}_1^{\frac{p-p}{r-p}} \right] \,,   
\end{equation}  
where $\text{A}_1=\displaystyle \int_{\R^N} |g||u|^p \ dx$ and $\text{B}_1= \displaystyle\int_{\R^N} |g||u|^{r} \ dx.$
Furthermore,
\begin{eqnarray*}
\displaystyle\lim_{t \ra 0} \frac{1}{t} \left[\text{A}_1^{\frac{r-(p+t)}{r-p}} \text{B}_1^{\frac{(p+t)-p}{r-p}} - \text{A}_1^{\frac{r-p}{r-p}} \text{B}_1^{\frac{p-p}{r-p}} \right] &=& \left(\frac{1}{r-p}\right) \text{A}_1  \log \left( \frac{ \text{B}_1}{\text{A}_1} \right) \,, \\
\displaystyle\lim_{t \ra 0} |g| \big[ \frac{|u|^{p+t}-|u|^p}{t}\big]
&=& \left(\frac{1}{p}\right) |g||u|^p \displaystyle \log \left( |u|^p \right)\,.
\end{eqnarray*}
By taking limit $t \ra 0$ in \eqref{mainineq2} and using Fatou's lemma we obtain
\begin{eqnarray} \label{41}
\displaystyle \int_{\R^N} |g||u|^p \displaystyle \log \left( |u|^p \right) \ dx &\leq& \left(\frac{p}{r-p}\right) \text{A}_1  \log \left( \frac{ \text{B}_1}{\text{A}_1} \right) \\
&=& \frac{r}{r-p} \text{A}_1  \log \left( \frac{ \text{B}_1^{\frac{p}{r}}}{\text{A}_1} \right) + \text{A}_1 \log \text{A}_1 \,. \nonumber
\end{eqnarray}
This gives,
\begin{eqnarray} \label{i}
\int_{\R^N} g |u|^p \log \left(\frac{|u|^p}{\int_{\R^N} g |u|^p}  \right) \ dx \leq \frac{r}{r-p} \left(\int_{\R^N} g |u|^p \ dx \right) \log \left(\frac{ \left[\displaystyle \int_{\R^N} |g||u|^{r} \ dx \right]^{\frac{p}{r}}}{\displaystyle\int_{\R^N} |g| |u|^p \ dx}  \right) \,.
\end{eqnarray}
Since $g \in \mathcal{H}_{p,r}(\R^N)$, it follows from Lemma \ref{Mazyalemma} that
\begin{eqnarray} \label{150}
\left[ \int_{\R^N} |g||u|^{r} \ dx \right]^{\frac{p}{r}} \leq \text{C}_H \|g\|_{p,r}^{\frac{p}{r}}  \int_{\R^N} |\nabla u|^p \ dx \,.
\end{eqnarray}
From \eqref{i} and \eqref{150}  we get
\begin{equation} \label{40}
 \int_{\R^N} g |u|^p \log (|u|^p) \ dx \leq \frac{r}{r-p} \log \left(\text{C}_H \|g\|_{p,r}^{\frac{p}{r}} \int_{\R^N} |\nabla u|^p \ dx \right)\,,
\end{equation}
where $\displaystyle\int_{\R^N} g |u|^p=1$.
This proves the Theorem \eqref{weightthm}.

\begin{example} \label{31} \rm
(A). For $r \in [p,p^*]$, consider the function $$g_1(x)=\frac{1}{|x|^{N-\frac{r}{p}(N-p)}} \ \text{in} \ \R^N \,.$$
It can be verified that $g_1 \in L^{\frac{p^*}{p^*-r},\infty}(\R^N).$ Hence, $g_1 \in \mathcal{H}_{p,r}(\R^N)$ (by Remark \ref{521}). Observe that $g_1 \notin L^{\frac{p^*}{p^*-r}}(\R^N)$. 
%Also, one can directly show that $\|g_1\|_{p,r} <\infty.$ For that, let $F \cset \R^N.$ Then it follows from Hardy-Littlehood inequality and P\'{o}lya-Szeg\H {o} inequality that 
%$$\displaystyle \frac{\int_{F} |g_1|}{[\text{Cap}_p(F)]^{\frac{r}{p}}}  \leq \displaystyle \frac{\int_{F^{\star}} |g_1^{\star}|}{[\text{Cap}_p(F^{\star})]^{\frac{r}{p}}},$$
%where $F^{\star}$ is the ball centered at origin that has same Lebesgue measure as $F$ and $g_1^{\star}(x):=g_1^*(\om_N|x|^N)$ is the Schwarz symmetrization of $g_1.$ $\om_N$ denotes the volume of unit ball in $\R^N$. Let $R$ be the radius of $F^{\star}$. Then $\text{Cap}_p(F^{\star})= N\om_N \left(\frac{N-1}{p-1}\right)^{p-1}R^{N-p}$, see Theorem 4.15 of \cite{EvansGa}. Also $g_1=g_1^{\star}$, since $g_1$ is radial and radially decreasing. Thus, the above inequality yields
%$$\displaystyle \frac{\int_{F} |g_1|}{[\text{Cap}_p(F)]^{\frac{r}{p}}}  \leq  \text{C(N,p,r)} \frac{R^{\frac{r}{p}(N-p)}}{R^{\frac{r}{p}(N-p)}}.$$
%This implies that $$\|g_1\|_{p,r} = \sup_{F \cset \R^N} \displaystyle \frac{\int_{F} |g_1|}{[\text{Cap}_p(F)]^{\frac{r}{p}}}< \infty.$$
%Hence, $g_1 \in \mathcal{H}_{p,r}(\R^N)$ and $L^{\frac{p^*}{p^*-r}}(\R^N) \subsetneq \mathcal{H}_{p,r}(\R^N).$ 

\noi (B). For  $r \in [p,p^*]$, let $g_2(z)=\frac{1}{|x|^{N\left(\frac{p^*-r}{p^*}\right)}}$ for $z=(x,y) \in \R^2\times \R^{N-2}$. By Theorem 2.1 of \cite{Tarantello} we have
$$\left[\int_{\R^N} g_2  |u|^r \ dz \right]^{\frac{1}{r}} \leq \text{C} \left[ \int_{\R^N} |\nabla u|^p \ dz \right]^{\frac{1}{p}}, \ \ \forall u \in \mathcal{D}^{1,p}_0(\R^N) \,.$$
Hence, $g_2 \in \mathcal{H}_{p,r}(\R^N)$ (by Lemma \ref{Mazyalemma}). 
\end{example}
\begin{rmk} \rm
$L^{\frac{p^*}{p^*-r},\infty}(\R^N) \subsetneq \mathcal{H}_{p,r}(\R^N)$ for $r \in [p,p^*]$. By Remark \ref{521} we have $L^{\frac{p^*}{p^*-r},\infty}(\R^N) \subseteq \mathcal{H}_{p,r}(\R^N)$. Now,
consider the function $g_2 \in \mathcal{H}_{p,r}(\R^N)$ in the above example. We show that 
$g_2 \notin L^{\frac{p^*}{p^*-r},\infty}(\R^N).$ On the contrary, if $g_2 \in L^{\frac{p^*}{p^*-r},\infty}(\R^N),$ then part (i) of Proposition \ref{LorentzProp} implies $g_2 \in L^{ \frac{p^*\delta}{p^*-r}}_{loc}(\R^N),$  $\forall \delta \in [ \frac{p^*-r}{p^*},  1)$. In that case, choosing $\delta > \max\{\frac{2}{N},\frac{p^*-r}{p^*}\}$ we obtain
\begin{eqnarray*}
 \int_{[-1,1]^N} g_2^{ \frac{p^*\delta}{p^*-r}} \ dz = \int_{[-1,1]^N} \frac{1}{|x|^{N\delta}} \ dz = 2^{N-2} \int_{[-1,1]^2} \frac{1}{|x|^{N\delta}} \ dx =\infty \,.
\end{eqnarray*}
This is a contradiction. Hence, $L^{\frac{p^*}{p^*-r},\infty}(\R^N) \subsetneq \mathcal{H}_{p,r}(\R^N)$. 
\end{rmk}

Recall that, for $r \in (p,p^*]$ and $\gamma \geq \frac{r}{r-p}$, the best constant in \eqref{weightedLSINEQ} is given by 
\begin{equation*} \label{bestcons}
\frac{1}{\text{C}_B(g,\gamma)} = \displaystyle \inf \left \{\frac{\int_{\R^N} |\nabla u|^p}{e^{\frac{1}{\gamma} \left(\int_{\R^N} |g||u|^p \log |u|^p \right)}} : u \in \mathcal{D}^{1,p}_0(\R^N), \int_{\R^N} |g||u|^p =1 \right \} \,.    
\end{equation*}

\noi In order to prove the Theorem \ref{Bestthm}, we first prove that the map $G(u):= \displaystyle \int_{\R^N} |g||u|^p$ on $\mathcal{D}^{1,p}_0(\R^N)$ is compact if $g \in \mathcal{F}_{p,p}(\R^N)$.
\begin{lemma} \label{cpct}
Let $g \in \mathcal{F}_{p,p}(\R^N)$ and $u_n \wra u$ in $\mathcal{D}^{1,p}_0(\R^N)$. Then
$$\displaystyle \lim_{n \ra \infty} \int_{\R^N} |g||u_n -u|^p \ dx = 0 \,.$$
\end{lemma}
\begin{proof} Let $\epsilon >0$ be arbitrary. Then there exists $g_{\var}\in C_c^{\infty}(\R^N)$ such that $\|g-g_{\var}\|_{p,p} < \epsilon.$ Now
\begin{eqnarray} \label{320}
\displaystyle \int_{\R^N} |g||u_n -u|^p \ dx \leq \int_{\R^N} |g-g_{\var}||u_n -u|^p \ dx + \int_{\R^N} |g_{\var}||u_n -u|^p \ dx \,.
\end{eqnarray}
Further, using Proposition \ref{associate} we obtain
$$\int_{\R^N} |g-g_{\var}||u_n -u|^p \ dx \leq  \|g -g_{\var}\|_{p,p} \|u_n-u\|_{p,p}' \leq \text{C} \epsilon \left[ \int_{\R^N} |\nabla (u_n-u)|\right]^{\frac{1}{p}} \,.$$
Since $u_n \wra u$ in $\mathcal{D}^{1,p}_0(\R^N)$, it follows that $\int_{\R^N} |\nabla (u_n-u)|^p$ is uniformly bounded and $\int_{\R^N} |g_{\var}||u_n -u|^p \ dx \ra 0$ as $n \ra \infty$ . Thus, \eqref{320} gives
$$\displaystyle \lim_{n \ra \infty} \int_{\R^N} |g||u_n -u|^p \ dx \leq \text{C}_1 \epsilon \,. $$
As $\epsilon >0$ is arbitrary, we prove the lemma.
\end{proof}
\noi {\bf{Proof of Theorem \ref{Bestthm}:}} Let $g \in \mathcal{H}_{p,r}(\R^N) \cap \mathcal{F}_{p,p}(\R^N)$ for some $r \in (p,p^*]$ and $\gamma > \frac{r}{r-p}$. Let $u_n \in \mathcal{D}^{1,p}_0(\R^N)$ be a minimising sequence of  $\frac{1}{\text{C}_B(g,\gamma)}$ i.e.
\begin{equation} \label{jj}
\frac{1}{\text{C}_B(g,\gamma)} = \lim_{n\ra \infty} \left[\frac{\int_{\R^N} |\nabla u_n|^p}{e^{\frac{1}{\gamma} \left(\int_{\R^N} |g||u_n|^p \log |u_n|^p \right)}}\right] \,.   \end{equation}
with $\int_{\R^N} |g||u_n|^p =1$.
We claim that $u_n$ is bounded in $\mathcal{D}^{1,p}_0(\R^N)$. By construction of $u_n$ we have
\begin{equation} \label{20}
\int_{\R^N} |\nabla u_n|^p \leq \frac{(1+\frac{1}{n})}{\text{C}_B(g,\gamma)}  \left[\displaystyle e^{\frac{1}{\gamma}\left(\int_{\R^N} |g||u_n|^p \log |u_n|^p \right)} \right] \,.
\end{equation}
It follows from \eqref{40} that
$$\displaystyle \int_{\R^N} |g||u_n|^p \log |u_n|^p \ dx \leq \frac{r}{r-p} \log \left( C \int_{\R^N} |\nabla u_n|^p  \right) \,.  $$
As a consequence, \eqref{20} gives
$$\int_{\R^N} |\nabla u_n|^p \leq \frac{2C_1}{\text{C}_B(g)}  \left[\int_{\R^N} |\nabla u_n|^p \right]^{\frac{r}{\gamma(r-p)}} \,.$$
Since $\frac{r}{\gamma(r-p)}<1$, $u_n$ is bounded in $\mathcal{D}^{1,p}_0(\R^N)$. Hence, $u_n \wra u$ in $\mathcal{D}^{1,p}_0(\R^N)$ upto a sub-sequence. Certainly, we have $\int_{\R^N} |\nabla u|^p \leq \lim_{n \ra \infty} \int_{\R^N} |\nabla u_n|^p$. Further, Lemma \ref{BLlemma} infers that
$$\displaystyle \lim_{n \ra \infty} \int_{\R^N} |g||u_n|^p \log |u_n|^p  = \int_{\R^N} |g||u|^p \log |u|^p +  \lim_{n \ra \infty} \int_{\R^N} |g||u_n-u|^p \log |u_n-u|^p \,.$$
Using this equality in \eqref{jj} we obtain
\begin{equation} \label{j}
\frac{1}{\text{C}_B(g,\gamma)} \geq \left[\frac{\int_{\R^N} |\nabla u|^p}{e^{\frac{1}{\gamma} \left(\int_{\R^N} |g||u|^p \log |u|^p \right)}}\right] \left[\frac{1}{e^{\frac{1}{\gamma} \left(\lim_{n \ra \infty}\int_{\R^N} |g||u_n-u|^p \log |u_n-u|^p \right)}} \right] \,.    
\end{equation}
Lemma \ref{cpct} ensures that $\int_{\R^N} |g||u|^p =1$ and hence, 
$$\frac{1}{\text{C}_B(g,\gamma)} \geq \frac{1}{\text{C}_B(g,\gamma)} \left[\frac{1}{e^{\frac{1}{\gamma} \left(\lim_{n \ra \infty}\int_{\R^N} |g||u_n-u|^p \log |u_n-u|^p \right)}} \right] \,.$$
This gives
$$\displaystyle \lim_{n \ra \infty}\int_{\R^N} |g||u_n-u|^p \log |u_n-u|^p \ dx \geq 0 \,. $$
On the other hand, it follows from \eqref{41}
\begin{equation*} 
\displaystyle \int_{\R^N} |g||u_n-u|^p \log |u_n-u|^p \ dx \leq \left[\frac{p}{r-p}\right] \left[ \text{A}_n \log \left( \int_{\R^N} |g||u_n-u|^r \right) - \text{A}_n \log \left(\text{A}_n \right)   \right] \,, 
\end{equation*}
where $\text{A}_n= \displaystyle \int_{\R^N} |g||u_n-u|^p \ dx$. Notice that $A_n \ra 0$ as $n \ra \infty$ (by Lemma \ref{cpct}). Further, since $u_n$ is bounded in $\mathcal{D}^{1,p}_0(\R^N)$, it follows that $\int_{\R^N} |g||u_n-u|^r$ is uniformly bounded. Thus, the above inequality yields
$$\displaystyle \lim_{n \ra \infty}\int_{\R^N} |g||u_n-u|^p \log |u_n-u|^p \ dx \leq 0 \,. $$
Therefore, $\displaystyle \lim_{n \ra \infty}\int_{\R^N} |g||u_n-u|^p \log |u_n-u|^p \ dx = 0$. Consequently, \eqref{j} implies
$$\frac{1}{\text{C}_B(g,\gamma)} \geq \left[\frac{\int_{\R^N} |\nabla u|^p}{e^{\frac{1}{\gamma} (\int_{\R^N} |g||u|^p \log |u|^p )}}\right] \,.$$
We also have $\int_{\R^N} |g||u|^p=1$. Hence, $u$ is a minimiser of $\frac{1}{\text{C}_B(g,\gamma)}.$

\begin{example} \rm
Let $p>2$. For $r \in [p,p^*]$, consider
\begin{eqnarray*}
 \tilde{g}_2(z)= \begin{cases*}
 \frac{1}{|x|^{N\left(\frac{p^*-r}{p^*}\right)}} \,, \quad \text{if} \ z \in [-1,1]^N \\ 
 0 \,, \quad \ \ \ \ \ \ \ \ \ \ \ \ \text{otherwise}
 \end{cases*}
\end{eqnarray*}
where $z=(x,y) \in \R^2\times \R^{N-2}$. We show that $\tilde{g}_2 \in \mathcal{H}_{p,r}(\R^N) \cap \mathcal{F}_{p,p}(\R^N)$. By part (B) of Example \ref{31}, $\tilde{g}_2 \in \mathcal{H}_{p,r}(\R^N).$ To show $\tilde{g}_2 \in \mathcal{F}_{p,p}(\R^N)$, we will use Theorem 31 of \cite{New}. Since support of $\tilde{g}_2$ is bounded, it can be easily seen that
$$\int_{\R^N} \tilde{g}_2 |u|^p \ dz \leq \text{C} \int_{\R^N} |\nabla u|^p \ dz, \ \ \forall u \in \mathcal{D}^{1,p}_0(\R^N) \,.$$
Hence, $\tilde {g}_2 \in \mathcal{H}_{p,p}(\R^N)$ (by Lemma \ref{Mazyalemma}).
Now, for any $z \in \R^N$, let $Q_{t}(z)$ be the cube of length $2t$ and centered at $z$. Then $B_{t}(z) \subseteq Q_t(z) \subseteq B_{\sqrt{N}t}(z)$. Since $\text{Cap}_p(B_{t}(z))=N\om_N\frac{N-1}{p-1}^{p-1} t^{N-p}$ (Theorem 4.15,\cite{EvansGa}), we have
$N\om_N\frac{N-1}{p-1}^{p-1} t^{N-p} \leq \text{Cap}_p(Q_t(z)) \leq N\om_N\frac{N-1}{p-1}^{p-1} (\sqrt{N}t)^{N-p}.$ Using this, for any $z \in [-1,1]^N$
and $t>0$, we obtain
\begin{equation*}
 \displaystyle \frac{\int_{Q_{t}(z)} \tilde{g}_2}{\text{Cap}_p(Q_{t}(z))} \leq C(N,p) \left[\frac{t^{N-2} \int_{[-1,1]^2} \tilde{g}_2 }{ t^{N-p}} \right] \,.   
\end{equation*}
Since $p>2$, the above inequality infers
\begin{equation} \label{62}
\lim_{t \ra 0} \left[\sup_{F \cset \R^N} \frac{\int_{F \cap Q_t(z)} \tilde{g}_2}{\text{Cap}_p(Q _{t}(z))} \right] =0, \ \ \ \forall z \in [-1,1]^N \,.   
\end{equation}
As $\tilde{g}_2$ vanishes outside $[-1,1]^N$, \eqref{62} holds for all $z \in \R^N$. For the same reason
\begin{equation} \label{63}
\lim_{t \ra \infty} \left[\sup_{F \cset \R^N} \frac{\int_{F \cap Q_t(0)^c} \tilde{g}_2}{\text{Cap}_p(Q_{t}(0)^c)} \right] =0 \,.    
\end{equation}
Since \eqref{62}, \eqref{63} hold and $\tilde{g}_2 \in \mathcal{H}_{p,p}(\R^N)$, it follows from Theorem 31 of \cite{New} that $\tilde{g}_2 \in \mathcal{F}_{p,p}(\R^N).$ 
\end{example}
\begin{rmk} \rm
$ L^1 \cap L^{\frac{p^*}{p^*-r}}(\R^N) \subsetneq \mathcal{H}_{p,r}(\R^N) \cap \mathcal{F}_{p,p}(\R^N).$ By part (ii) of Proposition \ref{functionspace}, we have $ L^1 \cap L^{\frac{p^*}{p^*-r}}(\R^N) \subseteq \mathcal{H}_{p,r}(\R^N) \cap \mathcal{F}_{p,p}(\R^N).$
The above example shows that $\tilde{g}_2 \in \mathcal{H}_{p,r}(\R^N) \cap \mathcal{F}_{p,p}(\R^N)$. But, following the similar computations as part (B) in Example \ref{31} one can easily verify that $\tilde{g}_2 \notin L^{\frac{p^*}{p^*-r}}(\R^N).$
\end{rmk}

\section{Logarithmic Hardy inequalities} \label{proofs}
In this section, we prove Theorem \ref{genthm}. The underlying idea is same as the proof of Theorem \ref{weightthm}. However, we give the proof for the sake of completeness. We also prove Theorem \ref{higherthm} in this section. 

\noi {\bf{Proof of Theorem \ref{genthm}:}}
For $q \in [p,p^*)$, take $k = p\frac{p^*-q}{p^*-p}$. Clearly, $k \in (0,p]$ and $k = N- \frac{(N-p)q}{p}$. Using Holder's inequality we estimate the following integral:
\begin{eqnarray*}
 \int_{\R^N} \frac{|u|^q}{\delta_{E}^{N-\frac{q}{p}(N-p-pa)}} \ dx &=& \int_{\R^N} \frac{|u|^k}{\delta_{E}^{k(a+1)}}  \frac{|u|^{q-k}}{\delta_{E}^{N-k(a+1)-\frac{q}{p}(N-p-pa)}} \ dx \\
 &\leq& \left[ \int_{\R^N} \frac{|u|^p}{\delta_{E}^{p(a+1)}} \ dx \right]^{\frac{k}{p}} \left[ \int_{\R^N} \frac{|u|^{\frac{p(q-k)}{p-k}}}{\delta_{E}^{[\frac{p}{p-k}][N-k(a+1)-\frac{q}{p}(N-p-pa)]}} \ dx \right]^{\frac{p-k}{p}} \\
 &= & \left[ \int_{\R^N} \frac{|u|^p}{\delta_{E}^{p(a+1)}} \ dx \right]^{\frac{p^*-q}{p^*-p}} \left[ \int_{\R^N} \frac{|u|^{p^*}}{\delta_{E}^{N-\frac{p^*}{p} [N-p-pa]}} \ dx \right]^{\frac{q-p}{p^*-p}} \,,
\end{eqnarray*}
for all $u \in C_c^{\infty}(\R^N).$
For small $t>0$, we take $q=p+t$ in the above inequality to obtain
\begin{eqnarray*} 
 \int_{\R^N} \frac{|u|^{p+t}}{\delta_{E}^{N-\frac{p+t}{p}(N-p-pa)}} \ dx
 \leq \left[ \int_{\R^N} \frac{|u|^p}{\delta_{E}^{p(a+1)}} \ dx \right]^{\frac{p^*-(p+t)}{p^*-p}} \left[ \int_{\R^N} \frac{|u|^{p^*}}{\delta_{E}^{N-\frac{p^*}{p} [N-p-pa]}} \ dx \right]^{\frac{(p+t)-p}{p^*-p}} \,.
\end{eqnarray*}
Notice that, for $t=0$ equality occurs. Thus, we obtain
\begin{equation} \label{mainineq}
 \displaystyle \int_{\R^N} \frac{1}{t} \left[\frac{|u|^{p+t}}{\delta_{E}^{N-\frac{p+t}{p}(N-p-pa)}}-\frac{|u|^p}{\delta_{E}^{N-\frac{p}{p}(N-p-pa)}} \right]
\leq \frac{1}{t} \left[\text{A}_1^{\frac{p^*-(p+t)}{p^*-p}} \text{B}_1^{\frac{(p+t)-p}{p^*-p}} - \text{A}_1^{\frac{p^*-p}{p^*-p}} \text{B}_1^{\frac{p-p}{p^*-p}} \right] \,,   
\end{equation}  
where $\text{A}_1=\displaystyle \int_{\R^N} \frac{|u|^p}{\delta_{E}^{p(a+1)}} \ dx$ and $\text{B}_1= \displaystyle\int_{\R^N} \frac{|u|^{p^*}}{\delta_{E}^{N-\frac{p^*}{p} [N-p-pa]}} \ dx.$
Furthermore,
\begin{eqnarray*}
\displaystyle\lim_{t \ra 0} \frac{1}{t} \left[\text{A}_1^{\frac{p^*-(p+t)}{p^*-p}} \text{B}_1^{\frac{(p+t)-p}{p^*-p}} - \text{A}_1^{\frac{p^*-p}{p^*-p}} \text{B}_1^{\frac{p-p}{p^*-p}} \right] &=& \left(\frac{1}{p^*-p}\right) \text{A}_1  \log \left( \frac{ \text{B}_1}{\text{A}_1} \right) \,, \\
\displaystyle\lim_{t \ra 0} \frac{1}{t} \left[\frac{|u|^{p+t}}{\delta_{E}^{N-\frac{p+t}{p}(N-p-pa)}}-\frac{|u|^p}{\delta_{E}^{N-\frac{p}{p}(N-p-pa)}} \right]
&=& \left(\frac{1}{p}\right)\frac{|u|^p}{\delta_E^{p(a+1)}} \displaystyle \log \left(\delta_E^{N-p-ap} |u|^p \right)\,.
\end{eqnarray*}
Hence, by taking limit $t \ra 0$ in \eqref{mainineq} and using Fatou's lemma, we get
\begin{eqnarray*}
\displaystyle \int_{\R^N} \frac{|u|^p}{\delta_E^{p(a+1)}} \displaystyle \log \left(\delta_E^{N-p-ap} |u|^p \right) \ dx &\leq& \left(\frac{p}{p^*-p}\right) \text{A}_1  \log \left( \frac{ \text{B}_1}{\text{A}_1} \right) \\
&=& \frac{p^*}{p^*-p} \text{A}_1  \log \left( \frac{ \text{B}_1^{\frac{p}{p^*}}}{\text{A}_1} \right) + \text{A}_1 \log \text{A}_1 \,.
\end{eqnarray*}
This yields
\begin{eqnarray*}
 \int_{\R^N} \frac{|u|^p}{\delta_E^{p(a+1)}} \displaystyle \log \left(\frac{\delta_E^{N-p-ap} |u|^p}{\int_{\R^N} \frac{|u|^p}{\delta_E^{p(a+1)}} } \right) \leq \frac{p^*}{p^*-p} \left[\int_{\R^N} \frac{|u|^p}{\delta_E^{p(a+1)}} \right] \log \left( \frac{\left[ \displaystyle \int_{\R^N} \frac{|u|^{p^*}}{\delta_{E}^{N-\frac{p^*}{p} [N-p-pa]}} \right]^{\frac{p}{p^*}}}{\int_{\R^N} \frac{|u|^p}{\delta_E^{p(a+1)}}} \right) \,.
\end{eqnarray*}
Now, since $a \in (-\frac{(N-d)(p-1)}{p},\frac{(N-p)(N-d)}{Np})$, we have
$$\text{dim}_A(E)=d<\frac{p^*}{p}(N-p-pa) \quad \mbox{and} \ \ \text{dim}_A(E)=d<N+\frac{pa}{p-1} \,.$$
Hence, it follows from Lemma \ref{Dydalemma} that
$$\left[\displaystyle \int_{\R^N} \frac{|u|^{p^*}}{\delta_{E}^{N-\frac{p^*}{p} [N-p-pa]}} \ dx \right]^{\frac{1}{p^*}} \leq \text{C} \left[\int_{\R^N} \frac{|\nabla u|^p}{\delta_E^{pa}} \ dx \right]^{\frac{1}{p}} \,.$$
Consequently,
\begin{eqnarray*}
 \int_{\R^N} \frac{|u|^p}{\delta_E^{p(a+1)}} \displaystyle \log \left(\frac{\delta_E^{N-p-ap} |u|^p}{\int_{\R^N} \frac{|u|^p}{\delta_E^{p(a+1)}} } \right) \ dx \leq \frac{N}{p} \left[\int_{\R^N} \frac{|u|^p}{\delta_E^{p(a+1)}} \ dx \right] \log \left( \frac{ \text{C} \int_{\R^N} \frac{|\nabla u|^p}{\delta_E^{pa}} \ dx}{\int_{\R^N} \frac{|u|^p}{\delta_E^{p(a+1)}} \ dx} \right) \,.
\end{eqnarray*}
By taking $\displaystyle \int_{\R^N} \frac{|u|^p}{\delta_E^{p(a+1)}}=1$ we obtain \eqref{GINTERDLSHINEQ}.

\begin{rmk} \rm
In particular, if we take $a=0$ in the above proof, then we obtain 
 \begin{equation} \label{l}
 \int_{\R^N} \frac{|u|^p}{|\delta_E(x)|^p} \log [|\delta_E(x)|^{N-p}|u|^p] \ dx \leq \frac{N}{p} \log \left( \text{C} \int_{\R^N}   |\nabla u|^p \ dx \right) \,, 
\end{equation}
for all $u \in C_c^{\infty}(\R^N)$ with $\displaystyle \int_{\R^N} \frac{|u|^p}{\delta_E^{p}}=1$.
Further, if we take $E=\{0\}$ and $p=2,$ then \eqref{l} coincides with \eqref{INTERLSHINEQ}.
\end{rmk}
\begin{rmk} \label{conseq}\rm
It follows from Remark \ref{boundaryporous}, that $\text{dim}_A(\partial B_1)< N$. By taking $a=0$ in Theorem \ref{genthm} we obtain an analogue of \eqref{l} on unit ball, namely
 \begin{equation*}
 \int_{B_1} \frac{|u|^p}{|\delta_{\partial B_1}(x)|^p} \log [|\delta_{\partial B_1}(x)|^{N-p}|u|^p]  \leq \frac{N}{p} \log \left( \text{C} \int_{B_1}   |\nabla u|^p \right) \,, 
\end{equation*}
for all $u \in C_c^{\infty}(B_1)$ with $\displaystyle \int_{B_1} \frac{|u|^p}{|\delta_{\partial B_1}(x)|^{p}}=1$.

\noi Remark \ref{boundaryporous} also shows that $\text{dim}_A(\R^{N-1} \times \{0\})< N$. Thus,  in a similar manner we get the following analogue of \eqref{l} in the half space
 \begin{equation*}
 \int_{\R^N_+} \frac{|u|^p}{|x_N|^p} \log [|x_N|^{N-p}|u|^p]  \leq \frac{N}{p} \log \left( \text{C} \int_{\R^N_+}   |\nabla u|^p \right) \,, 
\end{equation*}
for all $u \in C_c^{\infty}(\R^N_+)$ with $\displaystyle \int_{\R^N_+} \frac{|u|^p}{|x_N|^{p}}=1$.
 \end{rmk}
Next we proceed to prove Theorem \ref{higherthm} which is a second order generalisation of Theorem \ref{genthm}.

\noi {\bf{Proof of Theorem \ref{higherthm}:}} For $N \geq 3$ and $p \in (1,\frac{N}{2})$, let $E$ be a closed set in $\R^N$ with $\text{dim}_A(E) = d < \frac{N(N-2p)}{N-p} $ and $ a \in (1-\frac{(N-d)(p-1)}{p},\frac{(N-p)(N-d)}{Np})$. Then, proceeding as in the proof of Theorem \ref{genthm}, we obtain 
\begin{equation} \label{GDISTLOGHINTERSINEQ2}
 \int_{\R^N} \frac{|u|^p}{|\delta_E(x)|^{p(a+1)}} \log [|\delta_E(x)|^{N-p-pa}|u|^p] \ dx \leq \frac{N}{p} \log \left( \text{C} \int_{\R^N}   \frac{|\nabla u|^p}{\delta_E(x)^{ap}} \ dx \right) \,, 
\end{equation}
for all $u \in C_c^{\infty}(\R^N)$ with $\displaystyle \int_{\R^N} \frac{|u|^p}{\delta_E^{p(a+1)}}=1$. Now, since $ a \in (1-\frac{(N-d)(p-1)}{p},\frac{(N-p)(N-d)}{Np})$, one can see that
$$  d=\text{dim}_A(E)<N-p+(1-a)p \ \ \text{and}  \ \ d=\text{dim}_A(E)<N-\frac{(1-a)p}{p-1} \,.$$ Hence, Lemma \ref{Dydalemma} gives
\begin{eqnarray*}
 \int_{\R^N}   \frac{|\nabla u|^p}{\delta_E^{ap}} = \int_{\R^N}   \frac{|\nabla u|^p}{\delta_E^{N-\frac{p}{p}[N- p-(a-1)p]}} \leq C \int_{\R^N}   \frac{|\nabla^2 u|^p}{\delta_E^{(a-1)p}} \,,
\end{eqnarray*}
where $|\nabla^2 u|=\left(\sum_{i,j=1}^N |\frac{\partial^2 u}{\partial x_i \partial x_j}|^2 \right)^{\frac{1}{2}}.$ Consequently, we have \eqref{HIGHERINTERDLSHINEQ}.

\begin{rmk} \label{replap} \rm
 In particular, for $a=1$, \eqref{HIGHERINTERDLSHINEQ} corresponds to
 \begin{equation*} 
 \int_{\R^N} \frac{|u|^p}{|\delta_E(x)|^{2p}} \log [|\delta_E(x)|^{N-2p}|u|^p] \ dx  \leq \frac{N}{p} \log \left( \text{C} \int_{\R^N}  |\nabla^2 u|^p \ dx \right)\,, 
\end{equation*}
for all $u \in C_c^{\infty}(\R^N)$ with $\displaystyle \int_{\R^N} \frac{|u|^p}{\delta_E^{2p}}=1$, provided $\text{dim}_A(E)<\frac{N(N-2p)}{N-p}$. Further, if $E=\{0\}$ in the above inequality then
\begin{equation} \label{GDISTLOGHINTERSINEQ4}
 \int_{\R^N} \frac{|u|^p}{|x|^{2p}} \log [|x|^{N-2p}|u|^p] \ dx  \leq \frac{N}{p} \log \left( \text{C} \int_{\R^N}  |\nabla^2 u|^p \ dx \right) \,, 
\end{equation}
for all $u \in C_c^{\infty}(\R^N)$ with $\displaystyle \int_{\R^N} \frac{|u|^p}{|x|^{2p}}=1$. Notice that, this is a second order generalization of \eqref{INTERLSHINEQ}.
Recall that, for $p=2$, 
$$\int_{\R^N}  |\nabla^2 u|^2 \approx \int_{\R^N}  |\De u|^2, \ \ \forall u \in C_c^{\infty}(\R^N) \,.$$
Therefore, \eqref{GDISTLOGHINTERSINEQ4} yields
\begin{equation*}
 \int_{\R^N} \frac{|u|^2}{|x|^{4}} \log [|x|^{N-4}|u|^2] \ dx \leq \frac{N}{2} \log \left( \text{C} \int_{\R^N}  |\De u|^2 \ dx \right) \,, 
\end{equation*}
for all $u \in C_c^{\infty}(\R^N)$ with $\displaystyle \int_{\R^N} \frac{|u|^2}{|x|^{4}}=1$.
\end{rmk}

\section{Logarithmic Lorentz-Sobolev inequality} \label{4}
In this section, we prove Theorem \ref{LoSOthm}. The main underlying idea is adapted from the proof of Theorem \ref{weightthm}.

\noi{\bf{Proof of Theorem \ref{LoSOthm}:}} Let $q\in[p,p^*)$, we take $k= p\frac{p^*-q}{p^*-p}$. One can see that $\left[\frac{N}{p}-\frac{q}{p}[\frac{N}{p}-1]-\frac{k}{p}\right]=0$ and $\frac{p(q-k)}{p-k} =p^*.$ Then, using Holder's inequality we estimate
\begin{eqnarray*}
 \displaystyle \int_0^{\infty} s^{\left[\frac{p}{p^*}-1\right] \left[\frac{N}{p}-\frac{q}{p}(\frac{N}{p}-1)\right]} |u^*(s)|^q \ ds &=& \displaystyle \int_0^{\infty} s^{\frac{k}{p}\left[\frac{p}{p^*}-1\right]} |u^*(s)|^k s^{\left[\frac{p}{p^*}-1\right] \left[\frac{N}{p}-\frac{q}{p}(\frac{N}{p}-1)-\frac{k}{p}\right]} |u^*(s)|^{q-k} \\ 
 &\leq& 
\left[\displaystyle \int_0^{\infty} s^{[\frac{p}{p^*}-1]} |u^*(s)|^p \ ds \right]^{\frac{k}{p}} \left[\displaystyle \int_0^{\infty} |u^*(s)|^{p^*} \ ds \right]^{\frac{p-k}{p}} \,,
\end{eqnarray*}
for any $u \in \mathcal{D}^{1,p}_0(\R^N)$. Since $\displaystyle \int_0^{\infty} |u^*(s)|^{p^*} \ ds = \displaystyle \int_{\R^N} |u|^{p^*} \ dx$, the above inequality yields 
\begin{eqnarray*}
 \displaystyle \int_0^{\infty} s^{\left[\frac{p}{p^*}-1 \right] \left[\frac{N}{p}-\frac{q}{p}(\frac{N}{p}-1)\right]} |u^*(s)|^q \ ds \leq 
\left[\displaystyle \int_0^{\infty} s^{[\frac{p}{p^*}-1]} |u^*(s)|^p \ ds \right]^{\frac{k}{p}} \left[\displaystyle \int_{\R^N} |u|^{p^*} \ dx \right]^{\frac{p-k}{p}} \,.
\end{eqnarray*}
Now, following the same arguments as in the proof of Theorem \ref{weightthm} we get
$$\int_0^{\infty} s^{\left[\frac{p}{p^*}-1\right]} |u|^p \log (s^{1-\frac{p}{N}} |u|^p ) \ ds \leq \frac{N}{p} \log \left(\text{C} \int_{\R^N} |\nabla u|^p \ dx \right) \,, $$
for all $u \in \mathcal{D}^{1,p}_0(\R^N)$ with $\|u\|_{L^{p^*,p}}=1$.

\bibliography{ref}

\begin{thebibliography}{10}

\bibitem{Adams}
R.~A. Adams.
\newblock General logarithmic {S}obolev inequalities and {O}rlicz imbeddings.
\newblock {\em J. Functional Analysis}, 34(2):292--303, 1979.

\bibitem{Adimurthi}
Adimurthi, N.~Chaudhuri, and M.~Ramaswamy.
\newblock An improved {H}ardy-{S}obolev inequality and its application.
\newblock {\em Proc. Amer. Math. Soc.}, 130(2):489--505, 2002.

\bibitem{Allegretto}
W.~Allegretto.
\newblock Principal eigenvalues for indefinite-weight elliptic problems in
  {${\bf R}^n$}.
\newblock {\em Proc. Amer. Math. Soc.}, 116(3):701--706, 1992.

\bibitem{Alvino}
A.~Alvino, V.~Ferone, and G.~Trombetti.
\newblock On the best constant in a {H}ardy-{S}obolev inequality.
\newblock {\em Appl. Anal.}, 85(1-3):171--180, 2006.

\bibitem{anoop}
T.~V. Anoop.
\newblock A note on generalized {H}ardy-{S}obolev inequalities.
\newblock {\em Int. J. Anal.}, 2013:1--9, 2013.

\bibitem{New}
T.~V. Anoop and U.~Das.
\newblock The compactness and the concentration compactness via p-capacity.
\newblock {\em https://arxiv.org/abs/1905.06921}, 2018.

\bibitem{biharmonic}
T.~V. Anoop, U.~Das, and A.~Sarkar.
\newblock On the generalized \text{Hardy-Rellich} inequalities.
\newblock {\em Proceedings of the Royal Society of Edinburgh: Section A
  Mathematics}, 150(2):897–919, 2020.

\bibitem{Tarantello}
M.~Badiale and G.~Tarantello.
\newblock A {S}obolev-{H}ardy inequality with applications to a nonlinear
  elliptic equation arising in astrophysics.
\newblock {\em Arch. Ration. Mech. Anal.}, 163(4):259--293, 2002.

\bibitem{Barbatis}
G.~Barbatis, S.~Filippas, and A.~Tertikas.
\newblock Series expansion for $\text{L}^p$ hardy inequalities.
\newblock {\em Indiana Univ. Math. J.}, 52(1):171--190, 2003.

\bibitem{Bobkov}
S.~G. Bobkov and F.~G\"{o}tze.
\newblock Exponential integrability and transportation cost related to
  logarithmic {S}obolev inequalities.
\newblock {\em J. Funct. Anal.}, 163(1):1--28, 1999.

\bibitem{BrezisLeib}
H.~Br\'{e}zis and E.~a. Lieb.
\newblock A relation between pointwise convergence of functions and convergence
  of functionals.
\newblock {\em Proc. Amer. Math. Soc.}, 88(3):486--490, 1983.

\bibitem{CKN}
L.~Caffarelli, R.~Kohn, and L.~Nirenberg.
\newblock First order interpolation inequalities with weights.
\newblock {\em Compositio Math.}, 53(3):259--275, 1984.

\bibitem{Pick}
A.~Cianchi and L.~Pick.
\newblock Optimal {G}aussian {S}obolev embeddings.
\newblock {\em J. Funct. Anal.}, 256(11):3588--3642, 2009.

\bibitem{Dupaigne}
J.~D\'{a}vila and L.~Dupaigne.
\newblock Hardy-type inequalities.
\newblock {\em J. Eur. Math. Soc. (JEMS)}, 6(3):335--365, 2004.

\bibitem{Del}
M.~Del~Pino and J.~Dolbeault.
\newblock The optimal {E}uclidean {$L^p$}-{S}obolev logarithmic inequality.
\newblock {\em J. Funct. Anal.}, 197(1):151--161, 2003.

\bibitem{Pino}
M.~del Pino, J.~Dolbeault, S.~Filippas, and A.~Tertikas.
\newblock A logarithmic {H}ardy inequality.
\newblock {\em J. Funct. Anal.}, 259(8):2045--2072, 2010.

\bibitem{Dyda}
B.~Dyda, L.~Ihnatsyeva, J.~Lehrb\"{a}ck, H.~Tuominen, and A.~V.
  V\"{a}h\"{a}kangas.
\newblock Muckenhoupt {$A_p$}-properties of distance functions and applications
  to {H}ardy-{S}obolev--type inequalities.
\newblock {\em Potential Anal.}, 50(1):83--105, 2019.

\bibitem{EdEv}
D.~E. Edmunds and W.~D. Evans.
\newblock {\em Hardy operators, function spaces and embeddings}.
\newblock Springer Monographs in Mathematics. Springer-Verlag, Berlin, 2004.

\bibitem{EvansGa}
L.~C. Evans and R.~F. Gariepy.
\newblock {\em Measure theory and fine properties of functions}.
\newblock Textbooks in Mathematics. CRC Press, Boca Raton, FL, revised edition,
  2015.

\bibitem{Filippas}
S.~Filippas and A.~Tertikas.
\newblock Optimizing improved {H}ardy inequalities.
\newblock {\em J. Funct. Anal.}, 192(1):186--233, 2002.

\bibitem{Fraser}
J.~M. Fraser.
\newblock Assouad type dimensions and homogeneity of fractals.
\newblock {\em Trans. Amer. Math. Soc.}, 366(12):6687--6733, 2014.

\bibitem{Gkikas}
K.~T. Gkikas.
\newblock Hardy–sobolev inequalities in unbounded domains and heat kernel
  estimates.
\newblock {\em J. Funct. Anal.}, 264:837--893, 2013.

\bibitem{Gross}
L.~Gross.
\newblock Logarithmic {S}obolev inequalities.
\newblock {\em Amer. J. Math.}, 97(4):1061--1083, 1975.

\bibitem{Arthur}
S.~Holger and S.~Arthur.
\newblock Logarithmic sobolev inequalities for finite spin systems and
  applications.
\newblock {\em Bernoulli Society for Mathematical Statistics and Probability},
  95:109--132, 2001.

\bibitem{Holley}
R.~Holley and D.~Stroock.
\newblock Logarithmic {S}obolev inequalities and stochastic {I}sing models.
\newblock {\em J. Statist. Phys.}, 46(5-6):1159--1194, 1987.

\bibitem{Kufner}
A.~Kufner, L.~Maligranda, and L.~Persson.
\newblock {\em The {H}ardy inequality}.
\newblock Vydavatelsk\'{y} Servis, Plze\v{n}, 2007.
\newblock About its history and some related results.

\bibitem{Juha}
J.~Lehrb\"{a}ck and A.~V. V\"{a}h\"{a}kangas.
\newblock In between the inequalities of {S}obolev and {H}ardy.
\newblock {\em J. Funct. Anal.}, 271(2):330--364, 2016.

\bibitem{Lewis}
J.~L. Lewis.
\newblock Uniformly fat sets.
\newblock {\em Trans. Amer. Math. Soc.}, 308(1):177--196, 1988.

\bibitem{Lorentz}
G.~G. Lorentz.
\newblock Some new functional spaces.
\newblock {\em Ann. of Math. (2)}, 51:37--55, 1950.

\bibitem{Assouad}
J.~Luukkainen.
\newblock Assouad dimension: antifractal metrization, porous sets, and
  homogeneous measures.
\newblock {\em J. Korean Math. Soc.}, 35(1):23--76, 1998.

\bibitem{Malrieu}
F.~Malrieu.
\newblock Logarithmic sobolev inequalities for some nonlinear pde’s.
\newblock {\em Stochastic Processes and their Applications}, 95:109--132, 2001.

\bibitem{Milman}
J.~Mart\'{\i}n and M.~Milman.
\newblock Isoperimetry and symmetrization for logarithmic {S}obolev
  inequalities.
\newblock {\em J. Funct. Anal.}, 256(1):149--178, 2009.

\bibitem{Sobolevskii}
T.~Matskewich and P.~E. Sobolevskii.
\newblock The best possible constant in generalized {H}ardy's inequality for
  convex domain in {${\bf R}^n$}.
\newblock {\em Nonlinear Anal.}, 28(9):1601--1610, 1997.

\bibitem{Mazya2}
V.~G. Maz'ya.
\newblock Lectures on isoperimetric and isocapacitary inequalities in the
  theory of sobolev spaces.
\newblock {\em Contemp. Math.}, 338:307--340, 01 2003.

\bibitem{Opic}
B.~Opic and A.~Kufner.
\newblock {\em Hardy type inequalities}.
\newblock Harlow: Longman Scientific and Technical; New York: J. Wiley and
  sons, cop 1990.
\newblock Volume 219 of Pitman research notes in mathematics series.

\bibitem{Cattiaux}
C.~Patrick, A.~Guillin, and W.~Li-Ming.
\newblock Some remarks on weighted logarithmic sobolev inequality.
\newblock {\em Indiana University Mathematics Journal}, 60(6):1885--1904, 2011.

\bibitem{Bin}
B.~Qian and Z.~Zhang.
\newblock Weighted logarithmic sobolev inequalities for sub-gaussian measures.
\newblock {\em Acta Applicandae Mathematicae}, 116:1572--9036, 2011.

\bibitem{Junfang}
Y.~L. Roger T.~Lewis, Junfang~Li.
\newblock A geometric characterization of a sharp hardy inequality.
\newblock {\em J. Funct. Anal.}, 262:3159--3185, 2012.

\bibitem{Zographopoulos}
N.~M. Stavrakakis and N.~Zographopoulos.
\newblock Global bifurcation results for a semilinear biharmonic equation on
  all of {$\R^N$}.
\newblock {\em Z. Anal. Anwendungen}, 18(3):753--766, 1999.

\bibitem{Tartar}
L.~Tartar.
\newblock Imbedding theorems of sobolev spaces into lorentz spaces.
\newblock {\em Bollettino della Unione Matematica Italiana. Serie VIII. Sezione
  B. Articoli di Ricerca Matematica}, 1, 10 1998.

\bibitem{Tintarev}
A.~Tertikas and K.~Tintarev.
\newblock On existence of minimizers for the {H}ardy-{S}obolev-{M}azya
  inequality.
\newblock {\em Ann. Mat. Pura Appl. (4)}, 186(4):645--662, 2007.

\bibitem{Jesper}
J.~Tidblom.
\newblock Improved ${L}^p$ hardy inequalities.
\newblock {\em Doctoral thesis, Department of Mathematics, Stockholm
  University}, 2005.

\bibitem{Vazquez}
J.~L. Vazquez and E.~Zuazua.
\newblock The {H}ardy inequality and the asymptotic behaviour of the heat
  equation with an inverse-square potential.
\newblock {\em J. Funct. Anal.}, 173(1):103--153, 2000.

\bibitem{Visciglia}
N.~Visciglia.
\newblock A note about the generalized {H}ardy-{S}obolev inequality with
  potential in {$L^{p,d}(\R^n)$}.
\newblock {\em Calc. Var. Partial Differential Equations}, 24(2):167--184,
  2005.

\bibitem{Beckner}
M.~P. W.~Beckner.
\newblock On sharp sobolev embedding and the logarithmic sobolev inequality.
\newblock {\em Bull. London Math.Soc.}, 30:80--84, 1998.

\bibitem{Weissler}
F.~B. Weissler.
\newblock Logarithmic {S}obolev inequalities for the heat-diffusion semigroup.
\newblock {\em Trans. Amer. Math. Soc.}, 237:255--269, 1978.

\end{thebibliography}
\bibliographystyle{abbrv}

\noi {\bf	 Ujjal Das }\\  The Institute of Mathematical Sciences,\\ Chennai, 600113, India. \\
	 {\it Email}: ujjaldas@imsc.res.in, ujjal.rupam.das@gmail.com

\end{document}